\documentclass[12pt]{amsart}
\usepackage{tikz}
\usetikzlibrary{shapes.geometric, arrows.meta, positioning}
\usepackage{amsmath,amssymb,amsthm}
 \usepackage{float} 
 \usepackage{graphicx}
 \usepackage{booktabs} 
 \usepackage{float}
  \usepackage{xcolor}
 \usepackage[
  top=2cm,
  bottom=1.5cm,
  left=1.2in,
  right=1.2in
]{geometry}
\usepackage{multicol}
\usepackage{todonotes}
\usepackage[T1]{fontenc}
\usepackage[utf8]{inputenc}
\usepackage{lmodern}
\usepackage{url}
\usepackage{algorithm}       
\usepackage{algpseudocode}   
\usepackage{float}           

\newtheorem{proposition}{Proposition}[section]
\newtheorem{definition}{Definition}[section]
\newtheorem{theorem}{Theorem}[section]
\newtheorem{problem}{Problem}[section]
\title{\vspace{-1.5cm} 
    Magic property of fullerenes 
}

\author{Đorđe Barali\'{c}\and Adam Farhat}
\address{ \scriptsize{Mathematical Institute SASA, Belgrade, Serbia }}
\email{djbaralic@mi.sanu.ac.rs}

\address{\scriptsize{Dubai American Academy, United Arab Emirates}}
\email{adamfermata@gmail.com}

\begin{document}
\begin{abstract} 
Fullerenes are an allotrope of carbon having a hollow, cage-like structure. The atoms in the molecule are arranged in pentagonal and hexagonal rings such that each atom is connected to three other atoms. Simple polyhedra having only pentagonal and hexagonal faces are a mathematical model for fullerenes. We say that a fullerene with $n$ vertices has a magic property if the numbers $1, 2, \dots, n$ may be assigned to its vertices so that the sums of the numbers on each pentagonal face are equal and the sums of the numbers in each hexagonal face are equal. We show that $C_{8n+4}$  does not admit such an arrangement for all $n$, while there are fullerenes, like $C_{24}$ and $C_{26}$ that have many nonisomorphic such arrangements. 
\end{abstract}

\maketitle

\section{Introduction}

Forty years ago, Sir Harold W. Kroto, Robert F. Curl, Jr., and Richard Smalley discovered the first fullerene $C_{60}$, also known as buckminsterfullerene or "buckyball", see \cite{Kroto}. For this discovery, they were awarded the Nobel Prize for Chemistry in 1996.  The identification of fullerenes significantly broadened the range of recognized carbon allotropes, which had previously been restricted to graphite, diamond, and amorphous forms of carbon such as soot and charcoal. After the discovery of buckminsterfullerene, the existence of similar structures having 70, 76, 78, 82, 84, 90, 94, or 96 carbon atoms was confirmed. They have been the focus of extensive research, concerning both their chemical properties and their technological uses, particularly in materials science, electronics, and nanotechnology.

The experimental study of fullerenes was accompanied by theoretical investigations based on mathematical models of fullerene molecules called fullerene graphs. The vertices of the graphs are the atoms, and the edges are the bonds between the atoms in the molecules. Mathematically, a fullerene is a 3-connected 3-regular planar graph with only pentagonal faces and hexagonal faces. Equivalently, a mathematical fullerene may be regarded as a simple 3-polytope whose facets are pentagons or hexagons. Euler's formula implies that the number of pentagonal faces in a fullerene is 12 while the number of vertices is even. Gr\"{u}nbaum and Motzkin in \cite{Motzkin} showed that there exists a fullerene with any even $n\geq 24$  and with $n=20$ vertices, and that there is no fullerene with $n=22$ vertices. 

Barali\'{c} and Milenkovi\'{c} in \cite{permutohedron} proposed a study of an interesting property motivated by magic square configurations. They found an example of an arrangement of the first twenty four positive integers in across twenty four  vertices of the 3-dimensional permuthohedron satisfying that the sums of the numbers in the vertices of each square and each hexagonal facets are constant. They called this property \textit{magic}. The notion of a magic property can  be introduced similarly for fullerenes. A fullerene with $n$ vertices has a magic property if the first $n$ positive integers can be associated one per vertex so that the sum of the numbers in each pentagonal face is constant, as well as the respective sum corresponding to each hexagonal face. The aim of this article is to present and discuss some results about the magic property of fullerenes.  

\section{Fullerenes and the magic property}

Let us denote by $V=\left\{ {v}_1, {v}_2, \dots, {v}_n \right\}$ the set of vertices of a fullerene $C_n$. 

\begin{definition} Let $\mathcal{H}$ and $\mathcal{P}$ be the sets of hexagonal and pentagonal faces on a fullerene $C_n$. Fullerene $C_n$ is said to have the magic property if there is a bijection $f\colon V \to \{1, 2, \dots, n\}$, referred to as a magic configuration, such that
\[
\sum_{v \in H_i} f(v) = S_h, \quad \forall H_i \in \mathcal{H}
\]
\[
\sum_{v \in P_j} f(v) = S_p, \quad \forall P_j \in \mathcal{P}
\]
where the sum per pentagon $S_p$ and the sum per hexagon $S_h$ are deemed the magic constants of fullerene. 
\end{definition}

A fullerene may admit many distinct magic configurations, including with different magic constants, as we will see in the next section. However, the magic constants $S_p$ and $S_h$ for a given fullerene $C_n$ satisfy the next relation.

\begin{proposition}If a fullerene $C_n$ has a magic property, then
\begin{align}\label{veza}
24 S_p+(n-20) S_h=3n(n+1).
\end{align}
\end{proposition}

\begin{proof}
The relation follows by summing the numbers over all faces of $C_n$, which yields three times the sum of the first $n$ positive integers, since each vertex belongs to exactly three faces.
\end{proof}

Barali\'{c} and Milenkovi\'{c} used \eqref{veza} and a divisibility argument in \cite{permutohedron} to show that there is no magic configuration on the dodecahedron $C_{20}$. The same reasoning can be applied to the buckminsterfullerene $C_{60}$. If there were a magic configuration on $C_{60}$ then by $\eqref{veza}$ the magic constants would satisfy \[24 S_p+40 S_h=3 \cdot 60 \cdot 61,\] which is impossible due to the divisibility of the left-hand side by $8$. Indeed, the argument straightforwardly generalizes to the following result. 

\begin{theorem} 
If  $n \equiv 4\pmod{8}$ than a fullerene $C_n$ does not admit a magic configuration.
\end{theorem}

\begin{proof} Assume the contrary. Then $24 S_p+(n-20) S_h\equiv 0 \pmod{8}$. On the other hand $3 n(n+1)\equiv 4 \pmod{8}$. The contradiction!
    
\end{proof}

Based on congruences modulo $8$ we can deduce little about the magic constants in the remaining cases.

\begin{proposition} \[S_h\equiv \begin{cases}
			0 \pmod{2}, & \text{if $n\equiv 0 \pmod{8}$ }\\
            3 \pmod{4}, & \text{if $n\equiv \pm 2 \pmod{8}$}         
		 \end{cases}\]
\end{proposition}

Unfortunately, we cannot obtain much from utilizing these kinds of elementary number theory approach. Working modulo $3$, the most we can say is the following.

\begin{proposition}
If  $n \not\equiv 2\pmod{3}$ and a fullerene $C_n$ admits a magic configuration, then $S_h\equiv0 \pmod{3}$.
\end{proposition}

\section{Magic configurations on $C_{24}$ }

In the previous section, we established several nonexistence results for magic configurations on fullerenes. The simplest case of a fullerene is $C_{24}$: the Schlegel diagram for this fullerene is presented in Figure \ref{Figure1}. $C_{24}$ has $24$ vertices, twelve pentagonal and two hexagonal faces.

\begin{figure}[H]
    \centering
    \includegraphics[width=0.6\textwidth]{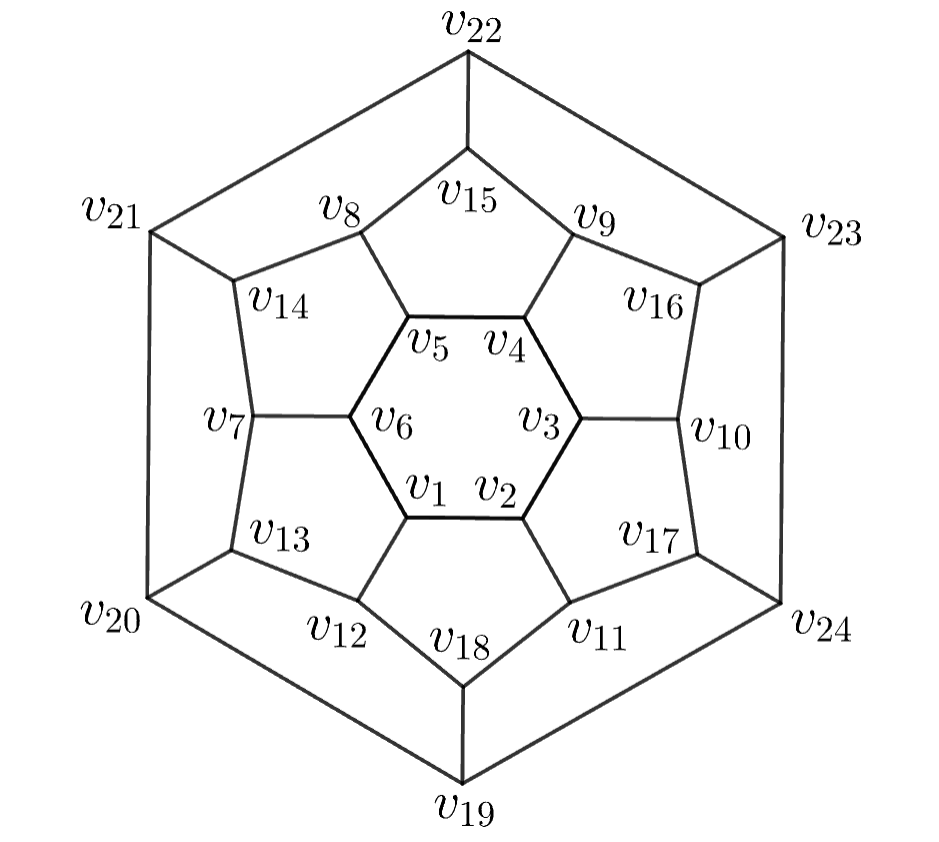} 
    \caption{Schlegel diagram of $C_{24}$}
    \label{Figure1} 
\end{figure}
We start the study of magic configurations on $C_{24}$ by examining its magic constants $S_p$ and $S_h$. Using \eqref{veza} we get

\begin{align} \label{veza24}
    6 S_p+S_h=450
\end{align}

The above linear Diophantine equation has solutions in positive integers, e.g. $S_p=50$ and $S_h=150$, so we cannot rule out the  possibility of a magic property of $C_{24}$. 

From \eqref{veza24} it is evident that $S_h\equiv 0 \pmod 6$. Two hexagonal $C_{24}$ do not share a vertex. The minimal possible value for $S_h$ is not less than the sum  when the numbers from 
 $\{1,2,3,4,\dots,12\}$ are placed on the hexagonal vertices, which is $39$. On the other hand $S_h$ is not greater than the sum when $\{13,14,15,16,\dots,24\}$ are placed on the hexagonal vertices, which is $111$.

However, $S_h\equiv 0 \pmod 6$ implies $S_h\in\{42,48,54,60,66,72,78,84,90,96,102,108\}$ and the corresponding values for $S_p$ are $S_p\in\{68, 67, 66,65,64,63,62,61,60,59,58, 57\}$. We exhausted all obvious number-theoretic arguments and no new constraints on the magic constants could be obtained. Attempting to search for a magic configuration with these constants by hand appears pointless. Therefore, we wrote a program \cite{github} to check which permutations of the first 24 positive integers satisfy the system \eqref{sis1} of Diophantine equations for each of the twelve possible pairs of magic constants $(S_p, S_h)$. The variables are associated with the vertices labeled in Figure \ref{Figure2}. More about the program will be given in Appendix A.  

\begin{equation}\label{sis1}
\begin{aligned}
v_1 + v_2 + v_3 + v_4 + v_5 + v_6 &= S_h &\quad
v_{19} + v_{20} + v_{21} + v_{22} + v_{23} + v_{24} &= S_h \\[0.3em]
v_5 + v_8 + v_{15} + v_9 + v_4 &= S_p &
v_4 + v_9 + v_{16} + v_{10} + v_3 &= S_p \\[0.3em]
v_3 + v_{10} + v_{17} + v_{11} + v_{2} &= S_p &
v_2 + v_{11} + v_{18} + v_{12} + v_1 &= S_p \\[0.3em]
v_1 + v_{12} + v_{13} + v_7 + v_6 &= S_p &
v_6 + v_7 + v_{14} + v_8 + v_5 &= S_p \\[0.3em]
v_{14} + v_8 + v_{15} + v_{22} + v_{21} &= S_p &
v_{15} + v_9 + v_{16} + v_{23} + v_{22} &= S_p \\[0.3em]
v_{16} + v_{10} + v_{17} + v_{24} + v_{23} &= S_p &
v_{17} + v_{11} + v_{18} + v_{19} + v_{24} &= S_p \\[0.3em]
v_{18} + v_{12} + v_{13} + v_{20} + v_{19} &= S_p &
v_{13} + v_7 + v_{14} + v_{21} + v_{20} &= S_p
\end{aligned}
\end{equation}

Surprisingly, our program found many  solutions in each of these twelve cases. 

\begin{table}[h]
\centering
\begin{tabular}{|cc|c|cc|}
\hline
$S_{p}$ & $S_{h}$ & \# of solutions & $S_{p}$ & $S_{h}$  \cr
\hline
57 & 108 & 576 & 68 & 42  \cr
58 & 102 & 936 & 67 & 48   \cr
59 & 96  & 2832 & 66 & 54   \cr
60 & 90  & 8832 & 65 & 60   \cr
61 & 84  & 11208 & 64 & 66   \cr
62 & 78  & 14592 & 63 & 72 \cr
\hline
\end{tabular}
\vspace{0.2 cm}
\caption{Number of magic configurations on $C_{24}$ with given magic constants}
\label{tab:solutions}
\end{table}

Figure \ref{Figure2} also illustrates the number of magic configurations on $C_{24}$. 

\begin{figure}[h]
    \centering
    \includegraphics[width=\textwidth]{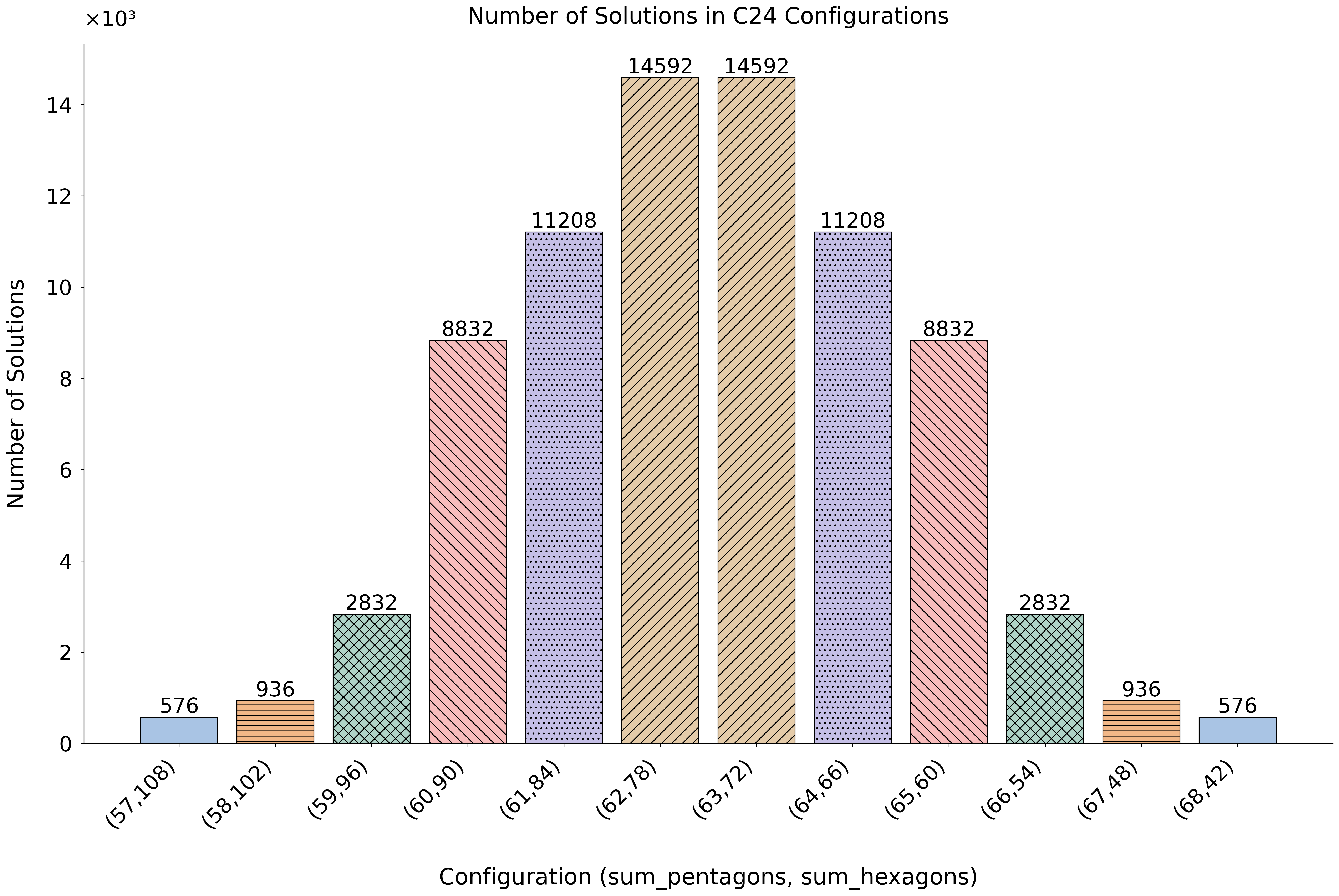} 
    \caption{Number of magic configurations on $C_{24}$ with given magic constants}
    \label{Figure2}
\end{figure}

 To conclude the analysis of $C_{24}$, we will provide examples of magic configurations on it. In Figure \ref{Figure3} we represent a magic configuration for each pair of twelve couples of the magic constants $S_p$ and $S_h$.

\begin{figure}[h]
    \centering
    \includegraphics[width=\textwidth]{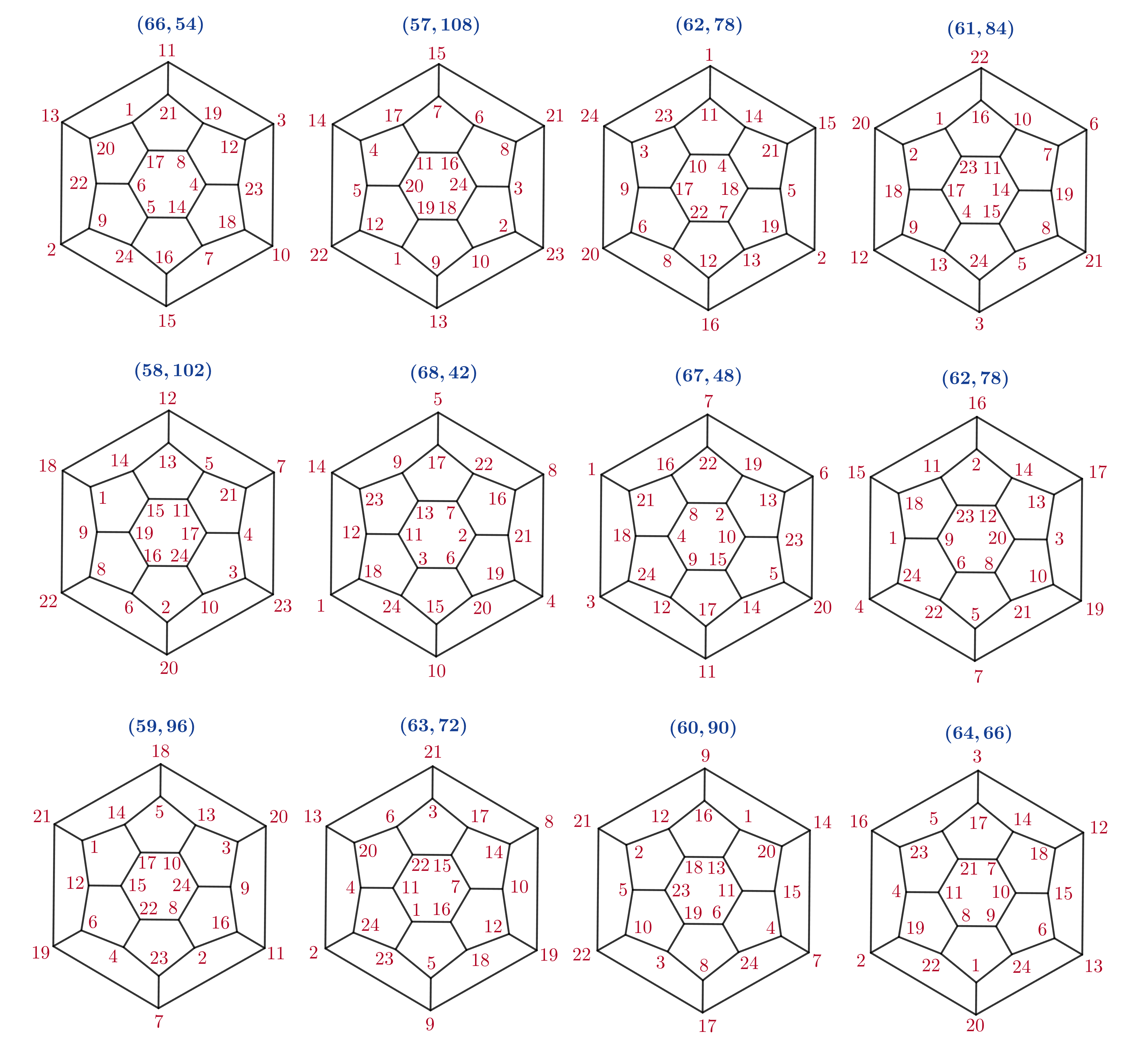} 
    \caption{Examples of magic configurations on $C_{24}$}\label{Figure3}
\end{figure}

\section{Magic Configurations on $C_{26}$ }

After analysing $C_{24}$, the next least complex fullerene is $C_{26}$, containing twelve pentagons, three hexagons, and 26 vertices. Its structure is shown in the Schlegel diagram in Figure \ref{Figure4}.

\begin{figure}[h]
    \centering
    \includegraphics[width=0.7\textwidth]{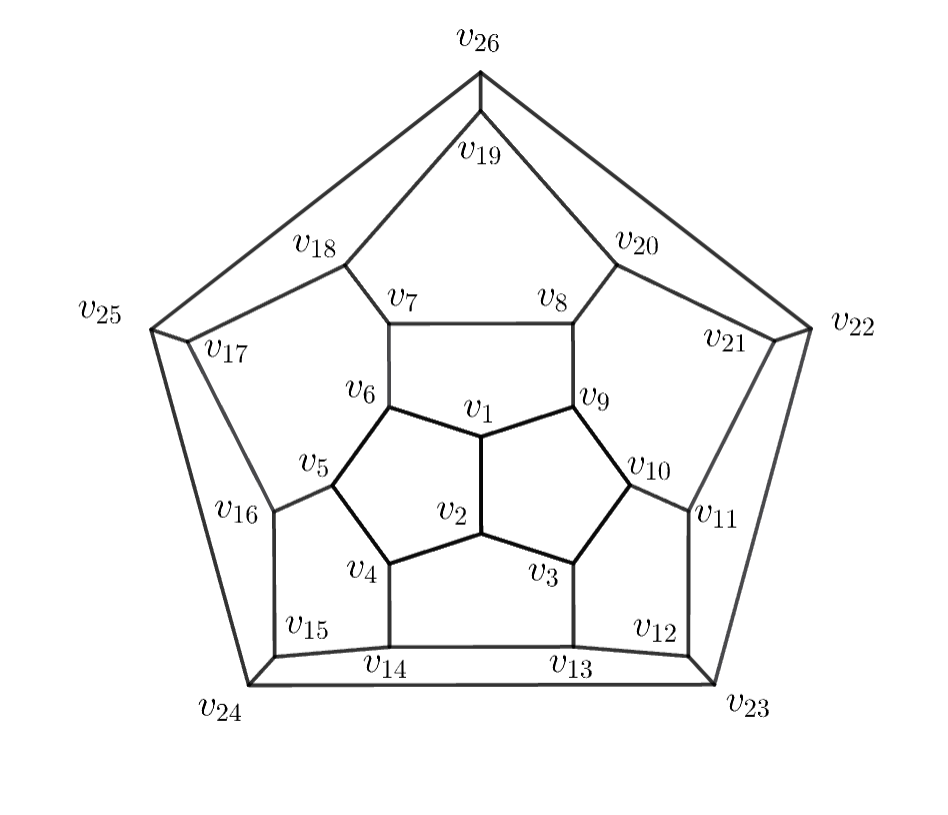} 
    \caption{Schlegel diagram of $C_{26}$}
    \label{Figure4}
\end{figure}

Similarly to $C_{24}$, we start studying the  magic configurations on $C_{26}$ by examining its magic constants $S_p$ and $S_h$. Using \eqref{veza} we get 
\begin{align} \label{veza26}
   12 S_p + 3 S_h = 1053
\end{align}

Since no two hexagonal faces of $C_{26}$ share a common vertex, using the same logic used in the analysis of $C_{24}$, the range of possible values for $S_h$ can be bounded by convention. That is,  the minimal value for $S_h$ cannot be less than the sum that arises when the smallest 18 integers $\{1,2,3,4,\dots,18\}$ are placed on the hexagonal vertices, yielding a sum of $57$. The same logic can be applied to find the maximum value of $S_h$ which is $105$. This exhausts all possible constraints for the magic constants of $C_{26}$ fullerene, and the Diophantine equation program must now be used to find all permutations of the first 26 numbers that are solutions for each of the 12 pairs of $(S_h,S_p)$.

\begin{equation}\label{sisC26}
\begin{aligned}
v_5 + v_6 + v_7 + v_{16} + v_{17} + v_{18} &= S_h 
&\quad v_{19} + v_{20} + v_{21} + v_{22} + v_{26} &= S_p \\[0.6em]
v_8 + v_9 + v_{10} + v_{11} + v_{20} + v_{21} &= S_h
&\quad v_{11} + v_{12} + v_{21} + v_{22} + v_{23} &= S_p \\[0.6em]
v_{12} + v_{13} + v_{14} + v_{15} + v_{23} + v_{24} &= S_h
&\quad v_3 + v_{10} + v_{11} + v_{12} + v_{13} &= S_p \\[0.6em]
v_1 + v_2 + v_4 + v_5 + v_6 &= S_p
&\quad v_2 + v_3 + v_4 + v_{13} + v_{14} &= S_p \\[0.6em]
v_1 + v_2 + v_3 + v_{10} + v_9 &= S_p
&\quad v_4 + v_5 + v_{14} + v_{15} + v_{16} &= S_p \\[0.6em]
v_1 + v_6 + v_7 + v_8 + v_9 &= S_p
&\quad v_{15} + v_{16} + v_{17} + v_{24} + v_{25} &= S_p \\[0.6em]
v_{17} + v_{18} + v_{19} + v_{25} + v_{26} &= S_p
&\quad v_{22} + v_{23} + v_{24} + v_{25} + v_{26} &= S_p \\[0.6em]
v_7 + v_8 + v_{18} + v_{19} + v_{20} &= S_p
\end{aligned}
\end{equation}

 In the above system, the variables correspond to the vertices marked in Figure \ref{Figure4}.

Our program \cite{github} found many solutions in each of these 12 cases, with the exception of $(73,59)$ and $(62, 103)$ where it found no solutions. Table \ref{t4} provides us with a summary of these solutions.

\begin{table}[H]
\centering
\begin{tabular}{|cc|c|cc|}
\hline
$S_p$ & $S_h$ & \# of Solutions & $S_p$ & $S_h$  \\
\hline
73 & 59  & 0 & 62 & 103 \\
72 & 63  & 84 & 63 & 99 \\
71 & 67  & 2796 & 64 & 95 \\
70 & 71  & 18228 & 65 & 91 \\
69 & 75  & 33252 & 66 & 87\\
68 & 79  & 58632 & 67 & 83 \\

\hline
\end{tabular}
\vspace{0.2cm}
\caption{Number of magic configurations on $C_{26}$ with given magic constants}\label{t4}
\end{table}

 The diagram in Figure \ref{Figure6} also represents the number of magic configurations on $C_{26}$ with given magic constants. 

\begin{figure}[H]
    \centering
    \includegraphics[width=\textwidth]{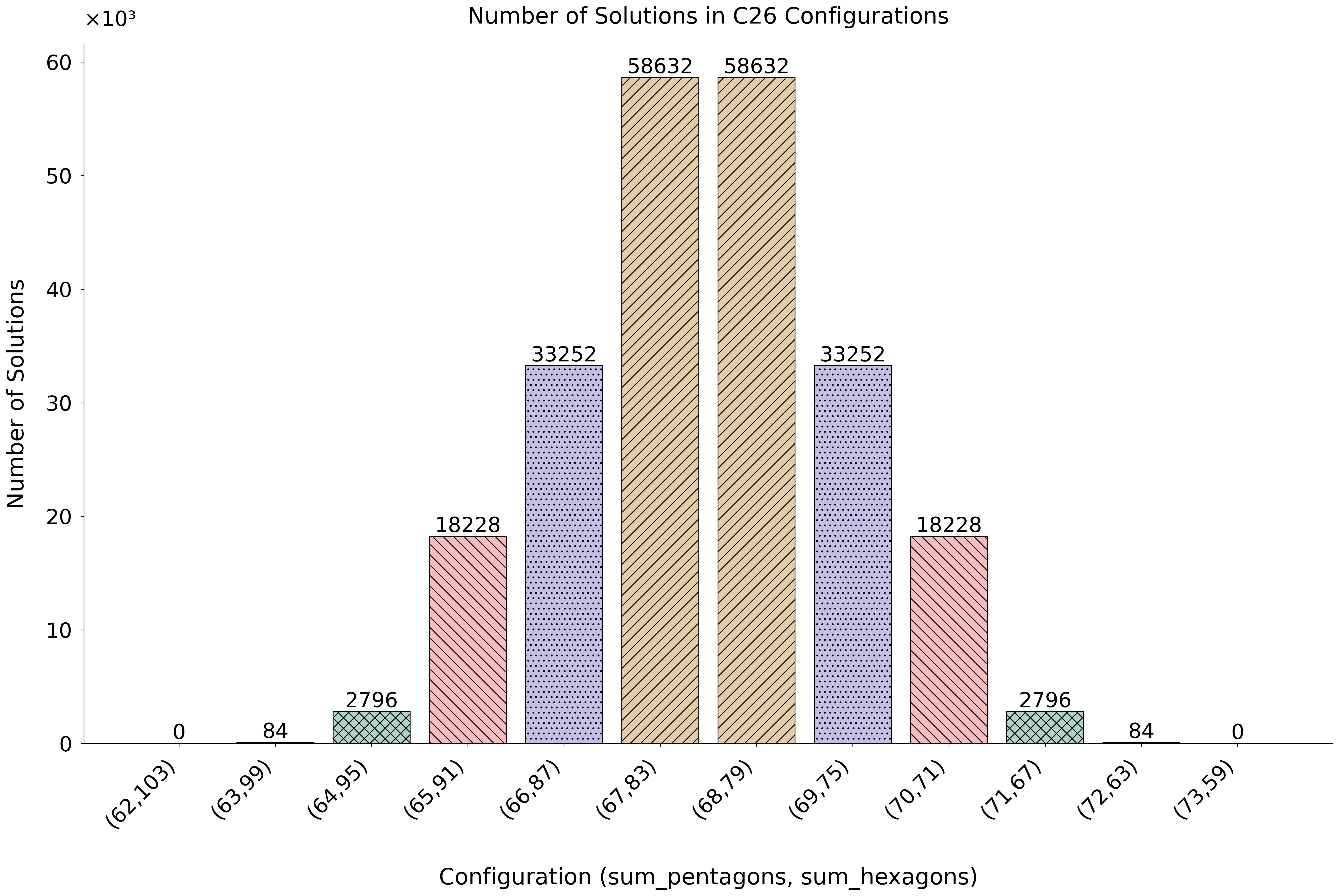} 
    \caption{Number of magic configurations on $C_{26}$ with given magic constants}
    \label{Figure6}
\end{figure}

In addition to this, we will provide examples of configurations on $C_{26}$ similar to that of the fullerene $C_{24}$. Figure \ref{Figure7} below displays a sample configuration on fullerene $C_{26}$ for 10 of the 12 pairs of magic constants, as constants $(73,59)$ and $(62,103)$ did not produce solutions.

\begin{figure}[H]
    \centering
    \includegraphics[width=\textwidth]{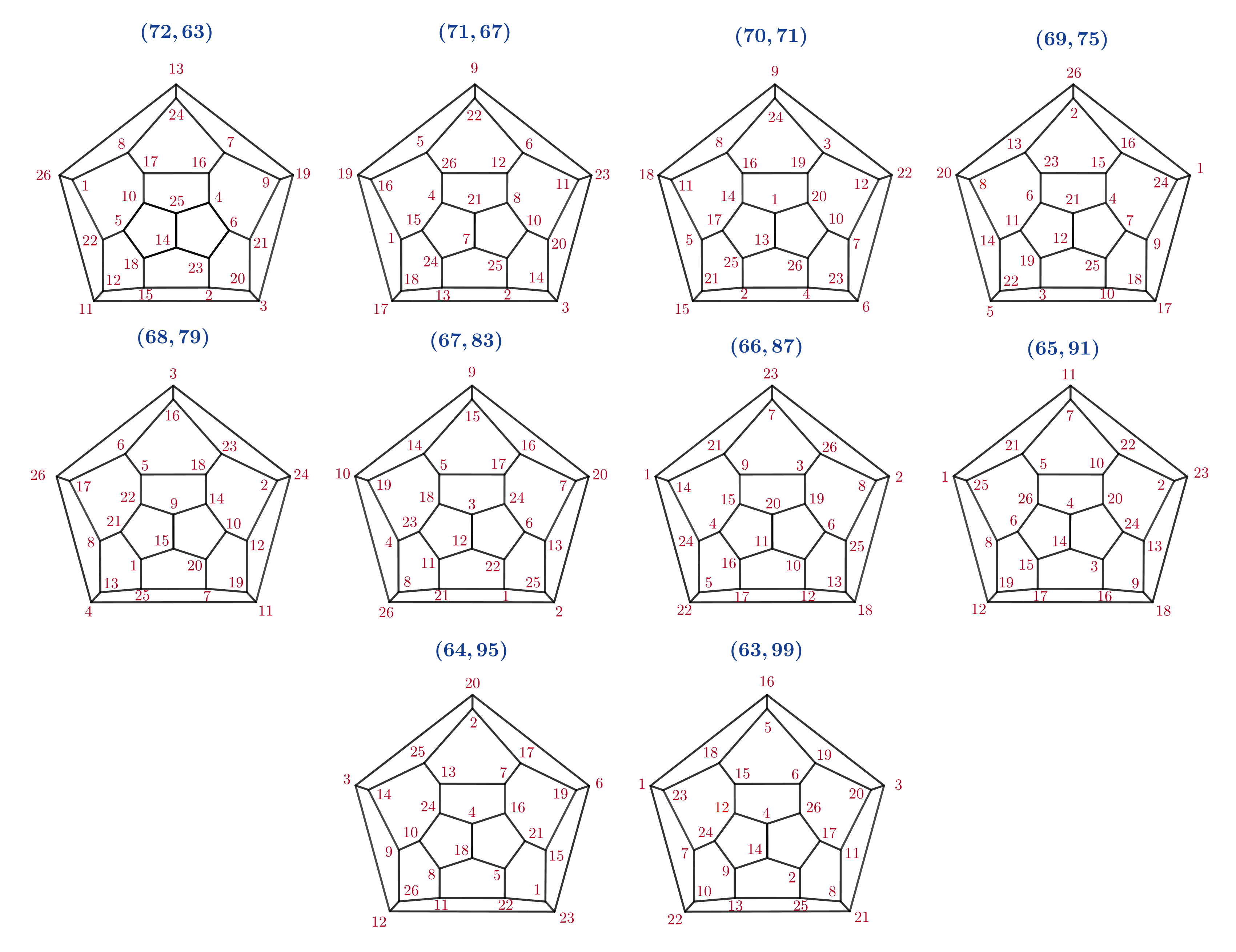} 
    \caption{Examples of magic configurations on $C_{26}$}
    \label{Figure7}
\end{figure}

\section{Symmetries of Fullerenes and Magic Configurations}

The study of the magic property of $C_{24}$ showed that a fullerene can have many magic configurations. Let us denote by $\mathcal{M}(C_n)$ the set of all magic configurations on $C_{n}$. $\mathcal{M}(C_n)$ is an interesting combinatorial structure associated with fullerene $C_n$, as will be seen in this section.

Assume that $V=\{v_1, v_2, \dots, v_n\}$ is the set of vertices of a fullerene $C_n$. An element of $\mathcal{M}(C_n)$ can be represented as a set $\{(v_1, f(v_1)),  (v_2, f(v_2)), \dots, (v_n, f(v_n))\}$ where $f\colon V\to \{1, 2, \dots, n\}$ is a magic configuration on $C_n$.

Let $G(C_n)$ be the automorphism group of the face poset of a fullerene $C_n$ and let $\theta \in G(C_n)$ be an automorphism of $C_n$. Then $\theta$ induces a map 
\begin{align}\label{dejstvo}
 \{(v_1, f(v_1)),   \dots, (v_n, f(v_n))\}\mapsto \{(\theta(v_1), f(v_1)),   \dots, (\theta(v_n), f(v_n))\}   
\end{align} We start  with the following apparent property.

\begin{proposition}\label{a1}
The group $G(C_n)$  acts freely on  $\mathcal{M}(C_n)$. 
\end{proposition}

\begin{proof} An automorphism of $G(C_n)$ sends the pentagons to the pentagons, and the hexagons to the hexagons. Therefore, a magic configuration on $C_n$ with magic constants $S_p$ and $S_h$ is sent to a magic configuration with the same constants.

According to \eqref{dejstvo} and the fact that any magical configuration $f$ is a bijection, to fix $f$, an automorphism  $\theta \in G(C_n)$ has to satisfy $\theta(v_i)=v_i$ for all $1\leq i \leq n$. But it means that $\theta$ fixes all vertices of $C_n$, so it must be a trivial element of $G(C_n)$. Therefore, the action is free.
\end{proof}

In addition to the action induced by $G(C_n)$, there exists another interesting action on  $\mathcal{M}(C_n)$.  

\begin{proposition}\label{a2}
A map $x\mapsto n+1-x$ induces a non-trivial map on  $\mathcal{M}(C_n)$. 
\end{proposition}

\begin{proof} Observe that a magic configuration on $C_n$ with magic constants $S_p$ and $S_h$ is sent to a magic configuration with $S^\prime_p=5n+5-S_p$ and $S^\prime_h=6n+6-S_h$.   
\end{proof}

Let us denote the induced map from Proposition \ref{a2} by $h$. Then $h$ sends a magic configuration $f$ to

\begin{align}\label{dejstvo2}
 \{(v_1, f(v_1)),  \dots, (v_n, f(v_n))\}\mapsto \{(v_1, n+1-f(v_1)),   \dots, (v_n, n+1-f(v_n))\}   
\end{align}

The composition $h \circ h$ is the identity map, so $h$ defines an $\mathbf{Z}_2$ action on $\mathcal{M}(C_n)$.

The following fact is a corollary of Propositions \ref{a1} and \ref{a2}.

\begin{theorem} The group $G(C_n)\oplus\mathbf{Z}_2$ acts freely on $\mathcal{M}(C_n)$.    
\end{theorem}

\begin{proof} Since $G(C_n)$ acts freely on $\mathcal{M}(C_n)$, we can assume that a magic configuration $f$ is fixed by $(\theta, h)$ for some automorphism $\theta \in G(C_n)$. However, by  \eqref{dejstvo} and \eqref{dejstvo2} $(\theta, h)$ sends $f$ with  magic constants $S_p$ and $S_h$
  to a magic configuration with $S^\prime_p=5n+5-S_p$ and $S^\prime_h=6n+6-S_h$. But since $n$ is even, it follows that $S^\prime_p\neq S_p$ contradicting the assumption that $f$ was fixed. Therefore, the action of $G(C_n)\oplus\mathbf{Z}_2$ on $\mathcal{M}(C_n)$ is free.   
\end{proof}

Since the automorphism group of $C_{24}$ is the dihedral group $D_{6}$ (\cite{deza} and \cite{Kutnar}), so the claim above is reflected in Table \ref{tab:solutions} by all numbers in the last column being divisible by 24. In this way, we obtain the number of non-isomorphic magic arrangements. 

The examples of $C_{24}$ and $C_{26}$ show that the magic configurations over a fullerene form an interesting mathematical structure that deserves further exploration. The next question to be addressed is:

\begin{problem} Determine all positive integers $n$ such that fullerene $C_n$ admits a magic configuration.
   \end{problem}
   
\section{Magic Configurations and Principal Component Analysis}

As we will see in this section, the set $\mathcal{M}(C_{24})$ provides interesting results when Principal Component Analysis (PCA) techniques are used to allow the solutions to be visualised in $\mathbb{R}^2$.

PCA is a statistical method of dimensionality reduction, often used with big and highly complex datasets \cite{JolliffeCadima}. The technique involves solving for new variables, which are linear functions of the original variables in the dataset, also known as the principal components. These principal components are the linear combinations of the original variables, determined by the eigenvectors of the data’s covariance matrix; the eigenvalues associated with each eigenvector are an indicator of how much variance each component explains \cite{JolliffeCadima}. In the case of  $\mathcal{M}(C_{24})$, we must first consider a magic configuration, a row in the dataset defined by variables $v_1$ to $v_{24}$, as a vector in $\mathbb{R}^{24}$. The first two principal components form a plane onto which the points are projected, yielding a 2-dimensional plot that is easily interpretable. This process and its mathematics are detailed below.

 PCA identifies dominant directions of variation in the solution set \( S \in \mathbb{R}^{N \times 24} \), where each row represents a labelling vector \( \mathbf{v}_i \in \mathbb{R}^{24} \).  
\noindent Centring is first performed by subtracting the empirical mean vector 
\[
\boldsymbol{\mu} = \frac{1}{N} \sum_{i=1}^N \mathbf{v}_i,
\]
to obtain the centred matrix \( \tilde{V} = V - \mathbf{1}_N \boldsymbol{\mu} \), where \( \mathbf{1}_N \) is an \( N \times 1 \) vector of ones.  
The covariance matrix is then computed as \( \Sigma = \frac{1}{N} \tilde{V}^\top \tilde{V} \), whose eigenvalue decomposition
\[
\Sigma = \sum_{i=1}^{24} \lambda_i \, \mathbf{u}_i \mathbf{u}_i^\top
\]
yields eigenvectors \( \mathbf{u}_i \) (principal directions) and eigenvalues \( \lambda_i \) (explained variances). From the definition of the covariance matrix, it follows that the matrix is symmetric and positive semi-definite, so all of its eigenvalues are non-negative real values.  

 The loading of vertex \( v \) on component \( i \) is defined as
\[
L_{v,i} = u_{v,i} \sqrt{\lambda_i},
\]
where \( u_{v,i} \) is the element of the eigenvector \( \mathbf{u}_i \) corresponding to vertex \( v \).  
Loadings quantify the extent to which each vertex contributes to a given principal component. Vertices with larger absolute loadings have a greater influence on that component’s variance and orientation.

To reduce dimensionality, the data is projected onto the first \( k \) principal components, with  
\[
W = 
\begin{bmatrix}
\mathbf{u}_1 & \mathbf{u}_2 & \dots & \mathbf{u}_k
\end{bmatrix}, \quad 
Z = \tilde{X} W.
\]
For visualization in \( \mathbb{R}^2 \), \( k = 2 \) is used, giving
\[
Z = \tilde{X} 
\begin{bmatrix}
\mathbf{u}_1 & \mathbf{u}_2
\end{bmatrix}
\in \mathbb{R}^{N \times 2},
\]
which represents each solution projected onto the plane spanned by the first two principal components.

In Table \ref{component_table} the eigenvalues for each pair of magic constants for $C_{24}$ are shown. The rank of \eqref{sis1} is $14$ and the nullity is therefore $10$ as it is indicated in Table \ref{component_table} with $14$ eigenvalues clustered around zero. It means that the set of the solutions of \eqref{sis1} is 10 10-dimensional affine subspace of $\mathbb{R}^{24}$. Therefore, an infinite set of lattice solutions is highly expected. Non-zero eigenvalues for each pair of magic constants appear in pairs of two same eigenvalues, confirming the action of the automorphism group of $G(C_{24})$  on the set of magical configurations for a fixed pair of magical constants. 

\begin{table}[h]
\centering

\tiny
\setlength{\tabcolsep}{5pt}
\begin{tabular}{|c|cc|cc|cc|cc|cc|cc|}
\hline
\textbf{PC} & \multicolumn{2}{c|}{\textbf{Pair 1}} & \multicolumn{2}{c|}{\textbf{Pair 2}} & \multicolumn{2}{c|}{\textbf{Pair 3}} & \multicolumn{2}{c|}{\textbf{Pair 4}} & \multicolumn{2}{c|}{\textbf{Pair 5}} & \multicolumn{2}{c|}{\textbf{Pair 6}} \\
\hline
 & \textbf{57,108} & \textbf{68,42} & \textbf{58,102} & \textbf{67,48} & \textbf{59,96} & \textbf{66,54} & \textbf{60,90} & \textbf{65,60} & \textbf{61,84} & \textbf{64,66} & \textbf{62,78} & \textbf{63,72} \\
\hline
1 & 57.05 & 57.05 & 87.32 & 87.32 & 114.61 & 114.61 & 114.04 & 114.04 & 121.66 & 121.66 & 121.28 & 121.28 \\
2 & 57.05 & 57.05 & 87.32 & 87.32 & 114.61 & 114.61 & 114.04 & 114.04 & 121.66 & 121.66 & 121.28 & 121.28 \\
3 & 50.24 & 50.24 & 66.39 & 66.39 & 90.26 & 90.26 & 105.24 & 105.24 & 120.07 & 120.07 & 116.06 & 116.06 \\
4 & 50.24 & 50.24 & 66.39 & 66.39 & 90.26 & 90.26 & 105.24 & 105.24 & 120.07 & 120.07 & 116.06 & 116.06 \\
5 & 36.23 & 36.23 & 65.58 & 65.58 & 84.35 & 84.35 & 101.40 & 101.40 & 107.31 & 107.31 & 112.10 & 112.10 \\
6 & 36.23 & 36.23 & 65.58 & 65.58 & 84.35 & 84.35 & 101.40 & 101.40 & 107.31 & 107.31 & 112.10 & 112.10 \\
7 & 34.91 & 34.91 & 64.54 & 64.54 & 81.01 & 81.01 & 94.45 & 94.45 & 104.14 & 104.14 & 111.67 & 111.67 \\
8 & 34.91 & 34.91 & 64.54 & 64.54 & 81.01 & 81.01 & 94.45 & 94.45 & 104.14 & 104.14 & 111.67 & 111.67 \\
9 & 33.93 & 33.93 & 48.52 & 48.52 & 57.92 & 57.92 & 84.93 & 84.93 & 94.86 & 94.86 & 110.93 & 110.93 \\
10 & 33.93 & 33.93 & 48.52 & 48.52 & 57.92 & 57.92 & 84.93 & 84.93 & 94.86 & 94.86 & 110.93 & 110.93 \\
11 & $\approx$0 & $\approx$0 & $\approx$0 & $\approx$0 & $\approx$0 & $\approx$0 & $\approx$0 & $\approx$0 & $\approx$0 & $\approx$0 & $\approx$0 & $\approx$0 \\
12 & $\approx$0 & $\approx$0 & $\approx$0 & $\approx$0 & $\approx$0 & $\approx$0 & $\approx$0 & $\approx$0 & $\approx$0 & $\approx$0 & $\approx$0 & $\approx$0 \\
13 & $\approx$0 & $\approx$0 & $\approx$0 & $\approx$0 & $\approx$0 & $\approx$0 & $\approx$0 & $\approx$0 & $\approx$0 & $\approx$0 & $\approx$0 & $\approx$0 \\
14 & $\approx$0 & $\approx$0 & $\approx$0 & $\approx$0 & $\approx$0 & $\approx$0 & $\approx$0 & $\approx$0 & $\approx$0 & $\approx$0 & $\approx$0 & $\approx$0 \\
15 & $\approx$0 & $\approx$0 & $\approx$0 & $\approx$0 & $\approx$0 & $\approx$0 & $\approx$0 & $\approx$0 & $\approx$0 & $\approx$0 & $\approx$0 & $\approx$0 \\
16 & $\approx$0 & $\approx$0 & $\approx$0 & $\approx$0 & $\approx$0 & $\approx$0 & $\approx$0 & $\approx$0 & $\approx$0 & $\approx$0 & $\approx$0 & $\approx$0 \\
17 & 0.00 & 0.00 & $\approx$0 & $\approx$0 & $\approx$0 & $\approx$0 & $\approx$0 & $\approx$0 & $\approx$0 & $\approx$0 & 0.00 & 0.00 \\
18 & 0.00 & 0.00 & 0.00 & 0.00 & $\approx$0 & $\approx$0 & $\approx$0 & $\approx$0 & $\approx$0 & $\approx$0 & 0.00 & 0.00 \\
19 & 0.00 & 0.00 & 0.00 & 0.00 & 0.00 & 0.00 & 0.00 & 0.00 & 0.00 & 0.00 & 0.00 & 0.00 \\
20 & 0.00 & 0.00 & 0.00 & 0.00 & 0.00 & 0.00 & 0.00 & 0.00 & 0.00 & 0.00 & 0.00 & 0.00 \\
21 & 0.00 & 0.00 & 0.00 & 0.00 & 0.00 & 0.00 & 0.00 & 0.00 & 0.00 & 0.00 & 0.00 & 0.00 \\
22 & 0.00 & 0.00 & 0.00 & 0.00 & 0.00 & 0.00 & 0.00 & 0.00 & 0.00 & 0.00 & 0.00 & 0.00 \\
23 & 0.00 & 0.00 & 0.00 & 0.00 & 0.00 & 0.00 & 0.00 & 0.00 & 0.00 & 0.00 & 0.00 & 0.00 \\
24 & 0.00 & 0.00 & 0.00 & 0.00 & 0.00 & 0.00 & 0.00 & 0.00 & 0.00 & 0.00 & 0.00 & 0.00 \\
\hline
\end{tabular}
\vspace{0.2 cm}
\caption{Eigenvalues of the covariance matrix for $C_{24}$}
\label{component_table}
\end{table}

 The projection for all 12 magic constants allowed by $\mathcal{M}(C_{24})$ is shown in Figure \ref{Figure8} below

\begin{figure}[H]
    \centering
    \includegraphics[width=1\textwidth]{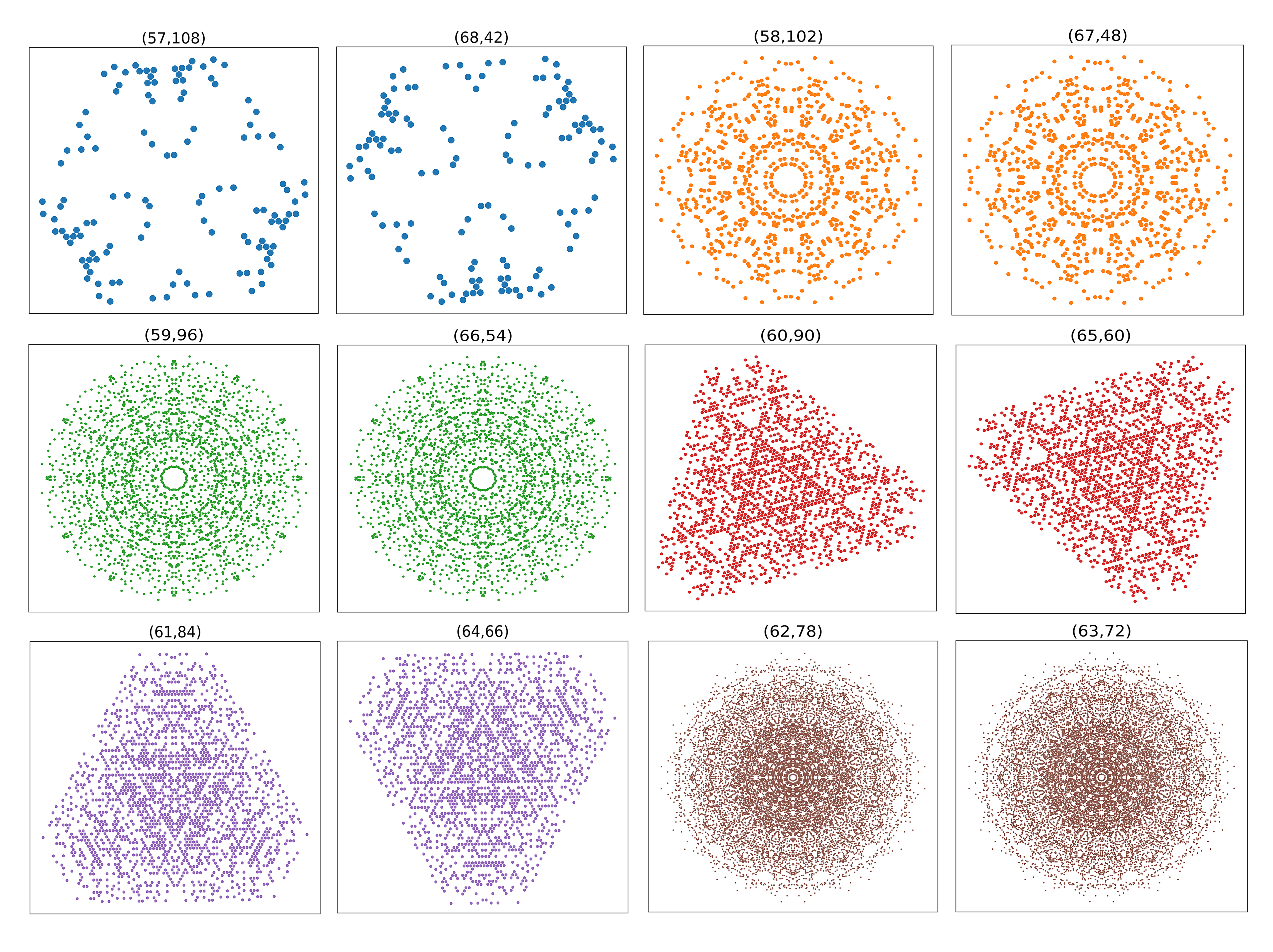} 
    \caption{}
    \label{Figure8}
\end{figure}

In Figure \ref{Figure8}, the PCA plots appear identical for different pairs of solutions that have the same number of solutions. This similarity arises because the solution sets share the same statistical structure. Specifically, the distribution of values across the vertices and the covariance structure of the assignments are alike. Since PCA focuses on patterns in variance, different configurations that produce the same covariance matrix—despite having different vertex values—lead to very similar PCA projections.

In addition to this observation, there also exists symmetrical structures within certain individual PCA plots in Figure \ref{Figure8}; Constants of (58,102), (59,96), and (62,78) provide circular plots with  12 axes of symmetry, while constants of (60,90) and (61,84) show triangular plots with three axes of symmetry. Precisely the dihedral group $D_{12}$ action can be noticed in the plots for (58, 102), (61, 84) and (62, 78) and the dihedral group $D_6 d$ in the plots for (57, 108), (60, 90) and (61, 84). While the presence of $D_6 d$ which is the automorphism group of $C_{24}$ is expectable, the appearance of $D_{12}$ is not unusual since $D_{12}$ and $D_6 d$ are isomorphic as abstract groups.

In the context of PCA methods, similar results may be expected for any fullerene $C_n$ admitting magical configurations. The nullity of the corresponding pentagonal and hexagonal faces for a pair of magical constants $(S_p, S_h)$ will be seen as the number of zero eigenvalues of the corresponding symmetric semi-definite covariance matrix  $\Sigma_{(S_p, S_h)}$. At the same time, the presence of the automorphism group $G(C_n)$ will produce multiplicities of certain eigenvalues. Also, $\mathbf{Z}_2$ action on $\mathcal{M}(C_{n})$ induced by \eqref{dejstvo2} produces the same PCA plots. 

\begin{proposition} Let $(S_p, S_h)$ and $(S'_p, S'_h)$ be magic constants for magic configurations on a fullerene $C_n$ such that $S_p+S'_p=5n+5$ and $S_p+S'_p=6n+6$. Then the corresponding covariance matrix  $\Sigma_{(S_p, S_h)}$ and $\Sigma_{(S'_p, S'_h)}$ are equal.    
\end{proposition}

\begin{proof} The covariance matrix $\Sigma_{(S_p, S_h)}$ is given by $\frac{1}{N} (V-\mathbf{1}_N \boldsymbol{\mu})^\top (V-\mathbf{1}_N \boldsymbol{\mu})$ where $V$ is $N\times n$ matrix of the solutions, \( \mathbf{1}_N \) is an \( N \times 1 \) vector of ones, and $\boldsymbol{\mu} = \frac{1}{N} \sum\limits_{i=1}^N \mathbf{v}_i$. The corresponding centred matrix of solutions for the pair $(S'_p, S'_h)$ is $$(n+1)\,  \mathbf{1}_N \times\mathbf{1}_n-V-((n+1) \mathbf{1}_N \times\mathbf{1}_n -\mathbf{1}_N \boldsymbol{\mu)}= \mathbf{1}_N \boldsymbol{\mu}-V, $$ where \( \mathbf{1}_n \) is an \( 1 \times n \) matrix of ones.  Thus, $$\Sigma_{(S'_p, S'_h)}= \frac{1}{N} (\mathbf{1}_N \boldsymbol{\mu}-V)^\top (\mathbf{1}_N \boldsymbol{\mu}-V)=\frac{1}{N} (V-\mathbf{1}_N \boldsymbol{\mu})^\top (V-\mathbf{1}_N \boldsymbol{\mu})=\Sigma_{(S_p, S_h)}, $$ as it was claimed.
    
\end{proof}

The set of magical configurations $\mathcal{M}(C_{n})$ is a unifying object for linear algebra, combinatorial design and group representations of $G(C_n)$, and it would be nice to say more about it.

\appendix 

\section{Program for Finding Magical Configurations}

Our program \cite{github} found all magical configurations on $C_{24}$ and $C_{26}$. The search for all valid magical labellings of \(C_{24}\) and \(C_{26}\) was formulated as a constraint satisfaction problem (CSP). Each vertex \(v_i\) was treated as an integer variable ranging from \(1\) to \(n\), subject to an \texttt{AllDifferent} constraint to ensure that every label appeared exactly once. Each pentagonal and hexagonal face introduced a linear constraint requiring the sum of its vertex labels to equal the respective pentagon or hexagon constant, \(S_p\) or \(S_h\).

Pseudocode below describes the algorithm’s search procedures: 

\begin{algorithm}[h]
\caption{Constraint-Based Search for Magic Labellings using CP-SAT}
\begin{algorithmic}[1]
\State \textbf{Input:} Fullerene type ($C_{24}$ or $C_{26}$), vertex count $n$, constants $(S_p, S_h)$
\State \textbf{Output:} All valid labellings $(v_1, v_2, \ldots, v_n)$ satisfying all face-sum constraints
\Statex

\State Define integer variables $v_i \in \{1, \ldots, n\}$
\State Apply \texttt{AllDifferent}($v_1, \ldots, v_n$) to ensures that all vertex labels are distinct.
\State Add face constraints:
\Statex \hspace{3em} $\sum_{v_i \in \text{pentagon}} v_i = S_p$
\Statex \hspace{3em} $\sum_{v_i \in \text{hexagon}} v_i = S_h$
\vspace{0.8em} 
\State Initialize CP-SAT solver with \texttt{enumerate\_all\_solutions = True} to find every possible labeling that satisfies all constraints.
\vspace{0.8em} 
\State \textbf{Note:} The solver explores possible assignments as a \textbf{search tree}, 
where each node represents a partial labelling and each branch corresponds to a variable choice.
\Statex

\Function{Search}{$partial\_assignment$}
    \If{all variables assigned and all constraints satisfied}
        \State Record solution to CSV
        \State \Return
    \EndIf
    \State Select next unassigned variable $v_k$
    \For{each feasible value $x$ in domain($v_k$)}
        \State Assign $v_k = x$
        \State Apply \textbf{constraint propagation} to update remaining domains
        \If{no constraint violated}
            \State \Call{Search}{$partial\_assignment \cup \{v_k = x\}$}
        \Else
            \State \textbf{Clause learning:} store this failed combination to skip later
        \EndIf
        \State Undo assignment of $v_k$ (\textbf{backtracking})
    \EndFor
\EndFunction
\Statex

\State \Call{Search}{$\emptyset$}
\Statex

\State \textbf{Solver mechanisms:}

\Statex \hspace{1.5em}\parbox[t]{0.9\linewidth}{
\textbf{Constraint propagation:} removes impossible values early.  
For example, for $([1,2,3,4,5,6], S_h)$, if $v_1, \dots, v_5$ are fixed, $v_6$ is determined automatically.
}

\Statex \vspace{0.3em}

\Statex \hspace{1.5em}\parbox[t]{0.9\linewidth}{
\textbf{Clause learning:} records conflicts like $(v_1,v_2,v_3)=(7,8,9)$ that violate a constraint, avoiding repetition.
}

\Statex \vspace{0.3em}

\Statex \hspace{1.5em}\parbox[t]{0.9\linewidth}{
\textbf{Backtracking:} when no feasible values remain, the solver returns to the previous variable and tries a new value.
}
\end{algorithmic}
\end{algorithm}

All admissible pairs \((S_p, S_h)\) were derived from \eqref{veza} and filtered under integer, range, and divisibility conditions before being used as model inputs. The model was implemented in Python~3.12 using the CP-SAT solver from Google OR-Tools, configured with \texttt{enumerate\_all\_solutions = True} to ensure complete enumeration of all valid labellings. The solver employs several exact, non-brute-force mechanisms that collectively optimise the search process. Constraint propagation continuously eliminates infeasible variable assignments by tightening domains as constraints interact. Clause learning records conflicts encountered during the search, preventing the solver from revisiting failed paths. Backtracking explores only feasible branches of the search tree, systematically returning to the last valid decision point when a contradiction is detected.  

These mechanisms reduce the theoretical search space from \(24^{24}\!\approx\!10^{33}\) combinations to roughly \(10^{6}\) feasible configurations while ensuring that no valid labelling is omitted. All results were saved in CSV format for reproducibility. Computations were performed on a dedicated workstation running Python~3.12, providing an exhaustive and exact enumeration of all configurations consistent with \eqref{veza} and the fullerene’s geometric constraints. 
\section*{Acknowledgments}

We are deeply thankful to the reviewers for their careful reading of our paper and their valuable feedback. Their keen insights and thoughtful suggestions were instrumental in refining our arguments and enhancing the clarity of our presentation. The first author was supported by  Project No. H20240855 of the Ministry of Human Resources and Social Security of the People's Republic of China, and by the Ministry of Science, Innovations and Technological Development of the Republic of Serbia.

\end{document}